\documentclass[12pt]{article}
\usepackage{mathptmx}
\usepackage{mathrsfs,amsthm}
\theoremstyle{definition}
\usepackage{url}
\usepackage{graphicx,psfrag,epsf,mathtools,dsfont,proof}
\usepackage{times}
\usepackage{fullpage}
\usepackage{setspace}
\usepackage{amsmath,amssymb,amsthm,amsfonts,amsbsy}
\usepackage{bm}
\usepackage{graphicx}
\usepackage{epsfig}
\usepackage{subfigure}
\usepackage{url}
\usepackage{psfrag}
\usepackage{multirow}
\usepackage{verbatim}
\usepackage{natbib}
\usepackage{color}    
\usepackage{booktabs}
\usepackage{graphicx,psfrag,epsf}
\usepackage{adjustbox}

\usepackage[title]{appendix}
\usepackage{authblk} 

\usepackage[colorlinks=true, citecolor=blue, urlcolor=blue]{hyperref}
\usepackage{natbib}

\DeclareFontFamily{U}{mathx}{\hyphenchar\font45}
\DeclareFontShape{U}{mathx}{m}{n}{
<5> <6> <7> <8> <9> <10>
<10.95> <12> <14.4> <17.28> <20.74> <24.88>
mathx10
}{}
\DeclareSymbolFont{mathx}{U}{mathx}{m}{n}
\DeclareFontSubstitution{U}{mathx}{m}{n}
\DeclareMathAccent{\widecheck}{0}{mathx}{"71}
\DeclareMathAccent{\wideparen}{0}{mathx}{"75}

\input cyracc.def

\newtheorem{proposition}{Proposition}
\newtheorem{theorem}{Theorem}

\newtheorem{corollary}{Corollary}

\title{\bf An upper bound and a characterization for Gini's mean  difference based on correlated random variables}

\author[1, 2]{Roberto Vila \thanks{rovig161@gmail.com}}
\author[2]{Narayanaswamy Balakrishnan  \thanks{bala@mcmaster.ca
} }
\author[1]{Helton Saulo  \thanks{heltonsaulo@gmail.com} }
\affil[1]{Department of Statistics, University of
	Bras\'ilia, Bras\'ilia, Brazil}
\affil[2]{
	Department of Mathematics and Statistics, McMaster University, Hamilton, Ontario, Canada}
\begin{document}

\maketitle

\begin{abstract}

In this paper, we obtain an upper bound for the Gini mean difference  
based on 
mean, variance and correlation for the case when the variables are correlated. 
We also derive
some 
closed-form expressions
for the Gini mean difference  when the random variables have an absolutely continuous joint distribution. 
We then examine some particular examples based on elliptically contoured distributions, and specifically multivariate normal and Student-$t$ distributions.
\\

\end{abstract}


\section{Introduction}\label{intro}
\noindent
Based on the random variables $X_1, \ldots, X_n$, the Gini mean difference (GMD) is defined as
\begin{align}\label{GMD-general}
GMD_n
=
\dfrac{1}{\displaystyle\binom{n}{2}} \sum_{1\leqslant i<j\leqslant n} {\mathbb{E}(\vert X_i-X_j\vert)},
\end{align}
provided the involved expectations exist. When $X_1, \ldots, X_n$ is a random sample (sequence of independent and identically distributed random variables), we obtain the classical GMD \cite[see][]{LaHaye19,  Schezhtman1987}, given by
\begin{align}\label{classic-GMD}
GMD=\mathbb{E}(\vert X_{1}-X_{2}\vert).
\end{align}

The GMD is a very useful measure of variability in the presence of non-normality.
In these cases, the GMD is superior to variance, in three different aspects: stochastic dominance, exchangeability and stratification; see \cite{Yitzhaki:03} and \cite{Yitzhaki:15} for details on other interesting features of the GMD. This measure has also been reported to be superior than Pearson and Spearman correlation coefficients for some distributions; see \cite{SchezhtmanYitzhaki1999} and \cite{Kattumannil:22}. Based on the GMD, \cite{SchmidSemeniuk2021} introduced tests on positive correlation and methods for monitoring the correlation structure of a process. Interesting applications of the GMD can also be found in survival analysis; as discussed by \cite{bonettietlal:09}, for example.


This work explores some properties of the GMD. First, in Section~\ref{sec:upper}, we derive an upper bound for the GMD which holds for all distributions with finite second moment. This upper bound is a refinement of the one derived by \citet[][]{Cerone:05}. In addition, the derived upper bound does not require independent random variables, as required in the works of \cite{Cerone:05} and \cite{Yin:22}. Next, in Section~\ref{sec:form}, we provide more informative results and present in particular methods for the exact calculation of GMD for jointly distributed  absolutely continuous random variables. Finally, in Section~\ref{sec:exam}, we derive the GMD for the multivariate normal and Student-$t$ distributions to illustrate the results, developed here.

\section{Upper bound for GMD}\label{sec:upper}

The following theorem presents an upper bound for the GMD that holds for all distributions with finite second moment.
\begin{theorem}\label{inicial-prop}
	Let $X_1,\ldots, X_n$ be random variables with finite second moments, and 
	let $\mu_i=\mathbb{E}(X_i)$, $\sigma_i^2={\rm Var}(X_i)$ and $\rho_{i,j}={\rm Corr}(X_i,X_j)$, for $i,j=1,\ldots,n$. Then, we have
	\begin{align*}
	GMD_n
	\leqslant
	\dfrac{1}{\displaystyle\binom{n}{2}} \sum_{1\leqslant i<j\leqslant n}
	\left[ 
	\sqrt{(\sigma_i-\sigma_j \rho_{i,j})^2+\sigma_j^2(1-\rho_{i,j}^2)}+\vert \mu_i-\mu_j\vert
	\right].
	\end{align*}
\end{theorem}
\begin{proof}
	Writing $\vert X_i-X_j\vert=\vert Y_i-Y_j +\mu_i-\mu_j\vert$, where $Y_i=X_i-\mu_i$ and $Y_j=X_j-\mu_j$, and then applying  triangular inequality, we get
	$\vert X_i-X_j\vert\leqslant \vert Y_i-Y_j\vert +\vert\mu_i-\mu_j\vert$.
	Now, by Jensen's inequality, we then have
	\begin{align*}
	\mathbb{E}(\vert X_i-X_j\vert)
	&\leqslant
	\sqrt{\mathbb{E}(\vert Y_i-Y_j\vert^2)}+\vert\mu_i-\mu_j\vert
	\\[0,2cm]
	&=
	\sqrt{{\rm Var}(Y_i-Y_j)}+\vert\mu_i-\mu_j\vert
	=
	\sqrt{{\rm Var}(X_i-X_j)}+\vert\mu_i-\mu_j\vert,
	\end{align*}
	where we have used the invariance with regard to translations of the variance in the last equality. Upon
	using the known identity that ${\rm Var}(X_i-X_j)=\sigma_i^2+\sigma_j^2-2{\rm Cov}(X_i,X_j)$, the expression on the right-hand side can be rewritten as
	\begin{align*}
	=\sqrt{\sigma_i^2+\sigma_j^2-2{\rm Cov}(X_i,X_j)}+\vert\mu_i-\mu_j\vert
	=
	\sqrt{\sigma_i^2+\sigma_j^2-2\sigma_i\sigma_j\rho_{i,j} }+\vert\mu_i-\mu_j\vert.
	\end{align*}
	Then, from the definition of the GMD in \eqref{GMD-general}, the proof gets completed.
\end{proof}

In Theorem \ref{inicial-prop}, upon taking
$\mu_1=\mathbb{E}(X_i)$, $\sigma_1^2={\rm Var}(X_i)$ 
and $\rho_{i,j}={\rm Corr}(X_i,X_j)$, for $i,j=1,\ldots,n$, we deduce 
\begin{align}\label{ineq-fund}
GMD_n
\leqslant
\sqrt{2}
\sigma_1\,
\dfrac{1}{\displaystyle\binom{n}{2}} \sum_{1\leqslant i<j\leqslant n} 
\sqrt{1- \rho_{i,j}}.
\end{align}
The above inequality has appeared in \citet[][Theorem 9.1]{SchmidSemeniuk2021}.
Furthermore, if we take $n=2$ in \eqref{ineq-fund}, we obtain
$GMD_2= {\mathbb{E}(\vert X_1-X_2\vert)}\leqslant \sqrt{2}
\sigma_1 \sqrt{1- \rho_{1,2}}$. So, for $\rho_{1,2}=0$, 
\begin{align}\label{ineq-GMD}
GMD_2\leqslant \sqrt{2}\sigma_1.
\end{align}
We emphasize that the upper bound in \eqref{ineq-GMD} refines the upper bound of $(4/\sqrt{3})\sigma_1$ in \citet[][]{Cerone:05}, but not the one of $(2/\sqrt{3})\sigma_1$  presented in
\citet[][Corollary 3.1]{Yin:22}. 
Another point worth mentioning is that the bound in \eqref{ineq-GMD} 
is for variables that 
%
are not necessarily independent, as required in the results of \cite{Cerone:05} and \cite{Yin:22}.

Let us now suppose that $X_1,\ldots, X_n,$ are independent and identically distributed  random variables with mean $\mu_1$ and variance $\sigma_1^2$, and that the variable $Z_1$, the standardized version of $X_1$, has a norm in $L^p$, denoted by $\lVert Z\rVert_p=[\mathbb{E}(\vert Z\vert^{p})]^{1/p}$, $p>1$. Moreover, let us denote $X_{(1)}=\min\{X_1,X_2\}$ and $X_{(2)}=\max\{X_1,X_2\}$ for the smallest and largest of the variables $X_1$ and $X_2$, respectively.
Then, from \cite{Barry:85}, it follows that
$
\mathbb{E}(X_{(2)})\leqslant \mu_1+\lVert Z_1\rVert_p
[({p-1})/({2p-1})]^{(p-1)/p} \sigma_1, \ p>1,
$
and consequently, the GMD in \eqref{classic-GMD} is such that
\begin{eqnarray}\label{ineq-GMD-1}
GMD=\mathbb{E}(X_{(2)}-X_{(1)})=2[\mathbb{E}(X_{(2)})-\mu_1]\leqslant C_p\sigma_1,
\end{eqnarray}
{with} 
$C_p=
2 \lVert Z_1\rVert_p [(p-1)/(2p-1)]^{(p-1)/p}.
$  
If we take $p=2$ in \eqref{ineq-GMD-1}, the upper bound $C_2 \sigma_1=(2/\sqrt{3})\sigma_1$ of \citet[][Corollary 3.1]{Yin:22} is obtained.
Further, depending on the choice of distribution of $X_1$, it is possible to find values of $1<p< 2$ such that the upper bound $C_p\sigma_1$ is slightly better than $(2/\sqrt{3})\sigma_1$.
%
For example, it is not difficult to verify that, for $X_1\sim N(0,1)$ and $p=3/2$, we have	
$
C_{3/2}\sigma_1
=
{[\Gamma({1/ 4})]^{2/3}}\sigma_1/({\sqrt{2} \sqrt[3]{\pi}}) 
\approx
1.14 \sigma_1
<
C_2\sigma_1=({2/ \sqrt{3}}) \sigma_1
\approx
1.15 \sigma_1.
$

\section{Closed-form for GMD}\label{sec:form}

In this section, we assume that the random variables $X_1,\ldots,X_n$ are absolutely continuous and are not necessarily independent. To state the main result of this section (in Theorem \ref{main-theorem}), we write
\begin{align}\label{def-h}
h_{i,j}(x)=f_{X_j}(x)\, {\pi_{i,j}(x)\over R_{i,j}},
\end{align}
where $\pi_{i,j}(x)=F_{X_i\vert (X_j=x)}(x)
$ denotes the conditional cumulative distribution function (CDF) of $X_i\vert (X_j=y)$ evaluated at $y=x$,
and $R_{i,j}=\mathbb{P}(X_i\leqslant X_j)=\mathbb{E}[\pi_{i,j}(X_j)]$ is the stress-strength reliability. 

\begin{proposition}\label{prop-inicial}
	The function  $h_{i,j}$ in \eqref{def-h} is a probability density function (PDF).
\end{proposition}
\begin{proof}
	A simple algebraic manipulation shows that
	\begin{align}
	h_{i,j}(x)
	=
	f_{X_i}(x)\,
	\dfrac{\pi_{j,i}(x)}{R_{j,i}}
	&=
	f_{X_i}(x)\,
	\dfrac{\mathbb{P}(X_i\geqslant X_j\vert X_i=x)}{R_{j,i}}
	\\[0,2cm]
	\nonumber
	&=
	f_{X_i}(x)\,
	\dfrac{\displaystyle\int_{0}^{\infty} f_{X_i-X_j\vert (X_i=x)}(y) {\rm d}y}{R_{j,i}}
	\nonumber 
	\\[0,2cm]
	&=
	\dfrac{\displaystyle\int_{0}^{\infty} f_{X_i,X_i-X_j}(x,y) {\rm d}y}{R_{j,i}}
	=
	f_{X_{i}\vert (X_i\geqslant X_j)}(x).
	\label{identidade-inicial}
	\end{align}
	So, from the above identities, it is evident that $h_{i,j}$ is indeed a PDF.
\end{proof}

Let $H_{i,j}$ be the CDF corresponding to $h_{i,j}$. From here on, we assume that all involved expected values exist.
\begin{theorem}\label{main-theorem}
	Let $\boldsymbol{X}=(X_1,\ldots,X_n)^\top$ be an absolutely continuous random vector. Then the GMD in \eqref{GMD-general} can be expressed as
	\begin{align*}
	GMD_n
	=
	\dfrac{1}{\displaystyle\binom{n}{2}} \sum_{1\leqslant i<j\leqslant n}
	(2R_{j,i}\mu_{H_{j,i}}+2R_{i,j}\mu_{H_{i,j}}
	-\mu_i-\mu_j),
	\end{align*}
	where $\mu_i=\mathbb{E}(X_i)$ and $\mu_{H_{i,j}}$ denotes the expected value with respect to the distribution $H_{i,j}$.
\end{theorem}
\begin{proof}
	Let  $m_{ij}=\min\{X_i,X_j\}$ and $M_{ij}=\max\{X_i,X_j\}$, $1\leqslant i<j\leqslant n$.
	Then the GMD in \eqref{GMD-general} can be written as
	\begin{align}\label{Gini-max-min}
	GMD_n
	=
	\dfrac{1}{\displaystyle\binom{n}{2}} \sum_{1\leqslant i<j\leqslant n}
	\mathbb{E}(M_{ij}-m_{ij}).
	\end{align}
	
	In order to find a closed-form expression for the GMD, from \eqref{Gini-max-min}, it is essential to know the PDFs of $M_{ij}$ and $m_{ij}$. Indeed, the law of total probability gives
	\begin{align*}
	F_{M_{ij}}(x)
	&=
	\mathbb{P}(X_i\leqslant x\vert X_i\geqslant X_j)R_{j,i}
	+
	\mathbb{P}(X_j\leqslant x\vert X_j \geqslant X_i)R_{i,j},
	\end{align*}
	where $R_{i,j}$ is as given in \eqref{def-h}.
	Then, by using \eqref{identidade-inicial}, the PDF of $M_{ij}$ is
	\begin{align}\label{pdf-2}
	f_{M_{ij}}(x)
	&=
	f_{X_{i}\vert (X_i\geqslant X_j)}(x) R_{j,i}
	+
	f_{X_{j}\vert (X_j\geqslant X_i)}(x) R_{i,j}
	\nonumber
	\\[,2cm]
	&
	=
	f_{X_i}(x) \pi_{j,i}(x)
	+
	f_{X_j}(x) \pi_{i,j}(x),
	\end{align}
	where $\pi_{j,i}$ is as defined in \eqref{def-h}.

	Now, upon using the well-known identities that $\min\{x,y\}=-\max\{-x,-y\}$, $f_{-X_i}(-x)=f_{X_i}(x)$ and $F_{-X_j\vert (-X_i=-x)}(-x)=1-\pi_{j,i}(x)$, 
	we can write the PDF of $m_{ij}=\min\{X_i,X_j\}$ as 
	\begin{align}\label{pdf-3}
	f_{m_{ij}}(x)
	&=
	f_{\max\{-X_i,-X_j\}}(-x)
	\nonumber
	\\[0,2cm]
	&\stackrel{\eqref{pdf-2}}{=}
	f_{-X_i}(-x) F_{-X_j\vert (-X_i=-x)}(-x)
	+
	f_{-X_j}(-x) F_{-X_i\vert (-X_j=-x)}(-x)
	\nonumber
	\\[0,2cm]
	&=	
	f_{X_i}(x) [1-\pi_{j,i}(x)]
	+
	f_{X_j}(x) [1-\pi_{i,j}(x)]
	\nonumber
	\\[0,2cm]
	&
	\stackrel{\eqref{pdf-2}}{=}
	f_{X_i}(x)+f_{X_j}(x)-f_{M_{ij}}(x).
	\end{align}
	
	Therefore, by \eqref{pdf-2} and \eqref{pdf-3}, the GMD  in \eqref{Gini-max-min} can be expressed as
	\begin{align}
	GMD_n
	&=
	\dfrac{1}{\displaystyle\binom{n}{2}} \sum_{1\leqslant i<j\leqslant n}
	\int_{-\infty}^{\infty} 
	x\,
	\big[f_{M_{ij}}(x)-f_{m_{ij}}(x)\big] {\rm d}x
	\nonumber
	\\[0,2cm]
	&=
	\dfrac{1}{\displaystyle\binom{n}{2}} 
	\sum_{1\leqslant i<j\leqslant n}
	\int_{-\infty}^{\infty} 
	x\,
	\big[2f_{X_i}(x) \pi_{j,i}(x)
	+
	2f_{X_j}(x) \pi_{i,j}(x)-f_{X_i}(x)-f_{X_j}(x)\big] {\rm d}x
	\label{Ginni-geral-formula-0}
	\\[0,2cm]
	&=
	\dfrac{1}{\displaystyle\binom{n}{2}} 
	\sum_{1\leqslant i<j\leqslant n}
	\int_{-\infty}^{\infty} 
	x\,
	\big[2R_{j,i}h_{j,i}(x) 
	+
	2R_{i,j} h_{i,j}(x) -f_{X_i}(x)-f_{X_j}(x)\big] {\rm d}x,
	\nonumber
	\end{align}
	where $R_{i,j}$ and $h_{i,j}$ are as given in \eqref{def-h}. This  completes the proof of theorem.
\end{proof}

A similar derivation of the PDF of the maximum of $n$ random variables has also been given by \cite{Arellano-Valle:08}.

In order to state the next result, we adopt the notation
$g_{i,j}(x)=2f_{X_j}(x) \pi_{i,j}(x)$,
where
$\pi_{i,j}(x)=F_{X_i\vert (X_j=x)}(x)$, $\pi_{i,j}$ is a skewing function, i.e., it satisfies $0\leqslant \pi_{i,j}(x)\leqslant 1$ and $\pi_{i,j}(-x)=1-\pi_{i,j}(x)$, and $f_{X_j}$ is symmetric around $0$. 
Under these conditions, it is clear that $g_{i,j}$
is a skew-symmetric PDF, and let
$G_{i,j}$ be the corresponding CDF.
\begin{corollary}
	Let $\boldsymbol{X}=(X_1,\ldots,X_n)^\top$ be an absolutely continuous random vector. 
	If $\pi_{i,j}$ is a skewing function and $f_{X_j}$ is symmetric around $0$, then
	the GMD is given by
	\begin{align*}
	GMD_n
	=
	\dfrac{1}{\displaystyle\binom{n}{2}} \sum_{1\leqslant i<j\leqslant n}
	(\mu_{G_{j,i}}+\mu_{G_{i,j}}
	-\mu_i-\mu_j),
	\end{align*}
	with $\mu_i=\mathbb{E}(X_i)$ and $\mu_{G_{i,j}}$ being the expected value with respect to the distribution $G_{i,j}$.
\end{corollary}
\begin{proof}
	The proof follows directly from the identity in \eqref{Ginni-geral-formula-0}.
\end{proof}

\begin{proposition}\label{corollary-exchangeable-case}
	Let $\boldsymbol{X}=(X_1,\ldots,X_n)^\top$	be an absolutely continuous exchangeable random vector. If $\pi_{1,2}$ is a skewing function and $f_{X_1}$ is symmetric around $0$, then
	the GMD is given by
	\begin{align*}
	GMD_n
	=
	\dfrac{2}{\displaystyle\binom{n}{2}} \sum_{1\leqslant i<j\leqslant n}
	(\mu_{G^*_{j,i}}-\mu_1),
	\end{align*}
	where $\mu_1=\mathbb{E}(X_1)$ and $\mu_{G^*_{j,i}}$ denotes the expected value with respect to the  skew-symmetric CDF $G^*_{j,i}(x)=\int_{-\infty}^x 2f_{X_i}(t)F_{X_j}(t){\rm d}t$.
\end{proposition}
\begin{proof}
	Because $\boldsymbol{X}$ is exchangeable, we have $(X_1,\ldots,X_n)^\top\stackrel{\mathscr{D}}{=}(X_{i_1},\ldots,X_{i_n})^\top$, for each permutation $(i_1\ldots ,i_n)$ of $(1,\ldots,n)$. Here, $\stackrel{\mathscr{D}}{=}$ means that
	the two sides of the equality have the same distribution. Consequently,
	$\mu_{G_{j,i}}=\mu_{G_{i,j}}$ and
	$\mu_i=\mu_j=\mu_1$, for $1\leqslant i<j\leqslant n$. Hence, by applying Corollary \ref{corollary-exchangeable-case}, the required result follows.
\end{proof}

As a consequence of Proposition \ref{corollary-exchangeable-case}, the following result readily follows.
\begin{proposition}\label{prop-iid}
	Let $X_1,\ldots,X_n$ be independent random variables with identical distributions $F_{X_1}$. Then, the classical GMD in \eqref{classic-GMD} is simply
	\begin{align}\label{prop-iid-0}
	GMD
	=
	\mu_{G_1}
	-\mu_1
	=
	\int_0^1 (2u-1) F_{X_1}^{-1}(u){\rm d}u,
	\end{align}
	where $\mu_1=\mathbb{E}(X_1)$ and $\mu_{G_1}$ is the expected value with respect to the  skew-symmetric CDF $G_1(x)=\int_{-\infty}^x 2f_{X_1}(t)F_{X_1}(t){\rm d}t$.
	
	Furthermore,  the Gini index, defined by $G=GMD/(2\mu_1)$, is simply 
	$
	G=\left[({\mu_{G_1}/ \mu_1})-1\right]/2
	$.
\end{proposition}

Observe that the formula in \eqref{prop-iid-0} has appeared recently in the work of \citet[][Proposition 2.1]{Yin:22}.

\section{Some examples}\label{sec:exam}

In this section, we consider the class of elliptically contoured (EC) distributions for the $n \times 1$ random vector $\boldsymbol{X}=(X_1,\ldots,X_n)^\top$. First, we recall that a random vector $\boldsymbol{X}$
has an EC distribution, also called an elliptically symmetric distribution, if $\boldsymbol{X}$ has joint PDF \cite[see][]{fkn:90} as
\begin{align*}
f(\boldsymbol{x})
=
k_n \vert \boldsymbol{\Sigma}\vert^{-1/2}
g^{(n)}[(\boldsymbol{x}-\boldsymbol{\mu})^\top\boldsymbol{\Sigma}^{-1} (\boldsymbol{x}-\boldsymbol{\mu})],
\quad \boldsymbol{x}\in\mathbb{R}^n.
\end{align*}
We use the notation $\boldsymbol{X} \sim EC_n(\boldsymbol{\mu}, \boldsymbol{\Sigma}, g^{(n)})$, where $\boldsymbol{\mu}\in\mathbb{R}^n$ is a location parameter,  $\boldsymbol{\Sigma}$ is a scale matrix (positive-definite real $n\times n$ matrix), $g^{(n)}$ is the PDF generator and $k_n$ is a normalizing constant.

As EC distributions are invariant under marginalization
and conditioning \cite[][Theorem 2.16]{fkn:90}, and our results only require marginal and conditional laws involving the variables $X_i$ and $X_j$, for $1\leqslant i<j\leqslant n$, our analysis is restricted to the marginalization
and conditioning corresponding to the bivariate case. In particular, we now discuss in detail the multivariate normal and Student-$t$ cases.

\subsection{Multivariate normal distribution}\label{Multivariate normal distribution}

Let $\boldsymbol{X}\sim EC_n(\boldsymbol{\mu}, \boldsymbol{\Sigma}, g^{(n)})$, where $g^{(n)}(x) = \exp(-x/2)$ is the PDF generator of
the multivariate normal distribution.
It is well-known that the conditional distribution of
$X_i$, given $X_j=x$, is 
$N(\mu_i+\rho_{i,j}(x-\mu_j)\sigma_i/\sigma_j,\sigma_i^2(1-\rho_{i,j}^2))$
and that its unconditional distribution is $X_j\sim N(\mu_j,\sigma_j^2)$. 
Let $\phi$ and $\Phi$ be the PDF and CDF of the standard normal distribution, respectively.
By using the standardization of $X_i\vert (X_j=x)$, we have
\begin{align*}
\pi_{i,j}(x)
=
F_{X_i\vert (X_j=x)}(x)
=
\Phi\left(
{1\over \sqrt{1-\rho_{i,j}^2}}\,
\left[
{{x-\mu_i\over\sigma_i}-\rho_{i,j} \left({x-\mu_j\over\sigma_j}\right)}
\right]
\right).
\end{align*}
Hence,  by \eqref{def-h},
\begin{align}\label{pdf-normal-h}
h_{i,j}(x)
=
{1\over R_{i,j}}\, 
{1\over \sigma_j}\, 
\phi\left(
{x-\mu_j\over\sigma_j}\right)
\Phi\left(
{1\over \sqrt{1-\rho_{i,j}^2}}\,
\left[
{{x-\mu_i\over\sigma_i}-\rho_{i,j} \left({x-\mu_j\over\sigma_j}\right)}
\right]
\right).
\end{align} 

By using \eqref{pdf-normal-h}  and making the change of variable $z=(x-\mu_j)/\sigma_j$, we have
\begin{align*}
\mu_{H_{i,j}}
=
\int_{-\infty}^\infty x {\rm d}H_{i,j}(x)
=
{1\over R_{i,j}}
\int_{-\infty}^\infty 
(\sigma_j z+\mu_j)\,
\phi(z)
\Phi\left(
{1\over \sqrt{1-\rho_{i,j}^2}}\,
\left[
{{\mu_j-\mu_i\over\sigma_i}
	+
	\left({\sigma_j\over\sigma_i}-\rho_{i,j}\right) z}
\right]
\right)
{\rm d}z.
\end{align*}
Using the well-known
formulas $\int_{-\infty}^\infty x\phi(x)\Phi(a+bx){\rm d}x=(b/\sqrt{1+b^2})\phi(a/\sqrt{1+b^2})$ and $\int_{-\infty}^\infty \phi(x)\Phi(a+bx){\rm d}x=\Phi(a/\sqrt{1+b^2})$,
the integral on the RHS of the above identity is 
\begin{align*}
=
{1\over R_{i,j}}
\left[
{\dfrac{\sigma_j}{c_{i,j}} \left(\dfrac{\sigma_j}{\sigma_i}-\rho_{i,j}\right)}\,
\phi\left(
\dfrac{1}{c_{i,j}}\left(\dfrac{\mu_j-\mu_i}{\sigma_i}\right)
\right)
+
\mu_j
\Phi\left(
\dfrac{1}{c_{i,j}}\left(\dfrac{\mu_j-\mu_i}{\sigma_i}\right)
\right)
\right],
\end{align*}
where
$
c_{i,j}=\sqrt{1-\rho_{i,j}^2+[({\sigma_j}/{\sigma_i})-\rho_{i,j}]^2}.
$

Thus, by applying Theorem \ref{main-theorem}, we obtain the following closed-form expression for the GMD:
\begin{align}\label{GMD-Gaussian-geral}
GMD_n
&=
\dfrac{2}{\displaystyle\binom{n}{2}} 
\sum_{1\leqslant i<j\leqslant n}
\left[
{\dfrac{\sigma_j}{c_{i,j}} \left(\dfrac{\sigma_j}{\sigma_i}-\rho_{i,j}\right)}\,
\phi\left(
\dfrac{1}{c_{i,j}}\left(\dfrac{\mu_j-\mu_i}{\sigma_i}\right)
\right)
+
\mu_j
\Phi\left(
\dfrac{1}{c_{i,j}}\left(\dfrac{\mu_j-\mu_i}{\sigma_i}\right)
\right)
-{\mu_j\over 2}
\right]
\nonumber
\\[0,2cm]
&+
\dfrac{2}{\displaystyle\binom{n}{2}} 
\sum_{1\leqslant i<j\leqslant n}
\left[
{\dfrac{\sigma_i}{c_{j,i}} \left(\dfrac{\sigma_i}{\sigma_j}-\rho_{j,i}\right)}\,
\phi\left(
\dfrac{1}{c_{j,i}}\left(\dfrac{\mu_i-\mu_j}{\sigma_j}\right)
\right)
+
\mu_i
\Phi\left(
\dfrac{1}{c_{j,i}}\left(\dfrac{\mu_i-\mu_j}{\sigma_j}\right)
\right)
-{\mu_i\over 2}
\right].
\end{align}

\begin{proposition}\label{GMD-gaussian}
	Let $\boldsymbol{X}\sim EC_n(\boldsymbol{\mu}, \boldsymbol{\Sigma}, g^{(n)})$ be an exchangeable random vector, where $g^{(n)}(x) = \exp(-x/2)$. 
	The GMD is then simply 
	\begin{align*}
	GMD_n
	=
	{2\over \sqrt{\pi}} \, {\sigma_1}\,
	\dfrac{1}{\displaystyle\binom{n}{2}} 
	\sum_{1\leqslant i<j\leqslant n}
	{  \sqrt{1-\rho_{i,j}}},
	\end{align*}
	where $\mu_1=\mathbb{E}(X_1)$ and $\sigma_1^2={{\rm Var}(X_1)}$.
\end{proposition}
\begin{proof}
	As $\boldsymbol{X}$ is an exchangeable vector, we have $\mu_i=\mu_j=\mu_1$, $\sigma_i=\sigma_j=\sigma_1$, $\rho_{i,j}=\rho_{j,i}$, $R_{i,j}=R_{j,i}$ and
	$c_{i,j}=c_{j,i}=\sqrt{2(1-\rho_{i,j})}$. Then, the required result follows directly from \eqref{GMD-Gaussian-geral}.
\end{proof}

The expression in Proposition \ref{GMD-gaussian} has appeared in the work  of \citet[][see below their Theorem 9.1]{SchmidSemeniuk2021}.

\subsection{Multivariate Student-$t$ distribution}\label{Multivariate Student's t-distribution}

Let $\boldsymbol{X}\sim EC_n(\boldsymbol{\mu}, \boldsymbol{\Sigma}, g^{(n)})$, where $g^{(n)}(x) = (1+x/\nu)^{-(\nu+n)/2}$ is the PDF generator of
the multivariate Student-$t$ distribution with $\nu$ degrees of freedom.
It is well-known that \cite[see Remark 1 of ][]{svcl:22}
\begin{align*}
X_i\vert (X_j=x)
\sim
t_{\nu+1}\Biggl(\mu_i+\rho_{i,j}\sigma_i\Big({x-\mu_j\over\sigma_j}\Big),{\nu+\left({x-\mu_j\over\sigma_j}\right)^2\over\nu+1}\, \sigma_i^2(1-\rho_{i,j}^2)\Biggr),
\end{align*}
and that its unconditional distribution is $X_j\sim t_\nu(\mu_j,\sigma_j^2)$. By 
using the standardization of $X_i\vert (X_j=x)$, we have 
\begin{align*}
\pi_{i,j}(x)
=
F_{X_i\vert (X_j=x)}(x)
&=
F_{\nu+1}
\left(
\sqrt{(\nu+1)/(1-\rho_{i,j}^2) \over \nu+\left({x-\mu_j\over\sigma_j}\right)^2}\,
\left[
{{x-\mu_i\over\sigma_i}-\rho_{i,j} \left({x-\mu_j\over\sigma_j}\right)}
\right]
\right).
\end{align*}
Hence, by \eqref{def-h}, we find
\begin{align}\label{pdf-student-h} 
h_{i,j}(x)
=
{1\over R_{i,j}}\,
{1\over \sigma_j}\,  
f_\nu\left(
{x-\mu_j\over\sigma_j}\right)
F_{\nu+1}
\left(
\sqrt{(\nu+1)/(1-\rho_{i,j}^2) \over \nu+\left({x-\mu_j\over\sigma_j}\right)^2}\,
\left[
{{x-\mu_i\over\sigma_i}-\rho_{i,j} \left({x-\mu_j\over\sigma_j}\right)}
\right]
\right).
\end{align} 
Here, $f_\nu$ and $F_\nu$, respectively, denote the PDF and CDF of a Student-$t$ distribution with $\nu$ degrees of freedom.

So, from the identity in \eqref{pdf-student-h}, by making the change of variable $z=(x-\mu_j)/\sigma_j$, it follows that
\begin{align}\label{exp-H}
\mu_{H_{i,j}}
&=
\int_{-\infty}^\infty 
x {\rm d}H_{i,j}(x)
=
{1\over R_{i,j}} 
\int_{-\infty}^\infty
(\sigma_j z+\mu_j)
f_\nu(z)
F_{\nu+1}
\left(
\sqrt{\nu+1 \over \nu+z^2}\, 
({\tau_{i,j}+\lambda_{i,j} z})
\right) 
{\rm d}z,
\end{align}
where 
$\tau_{i,j}=
(\mu_j-\mu_i)\big/\big(\sigma_i\sqrt{1-\rho_{i,j}^2}\,\big)$ 
and 
$\lambda_{i,j}=
[(\sigma_j/\sigma_i)-\rho_{i,j}]\big/\sqrt{1-\rho_{i,j}^2}$. Multiplying and dividing by $F_\nu\big(\tau_{i,j}\big/\sqrt{1+\lambda_{i,j}^2}\,\big)$, $\mu_{H_{i,j}}$ in \eqref{exp-H} can be expressed as
\begin{align}\label{int-uH}
\mu_{H_{i,j}}
=
{F_\nu\left(\tau_{i,j}\Bigl/ \sqrt{1+\lambda_{i,j}^2}\, \right)\over R_{i,j}} 
\int_{-\infty}^\infty
(\sigma_j z+\mu_j)
f_{\nu}(z;\lambda_{i,j},\tau_{i,j})
{\rm d}z,
\end{align}
with
\begin{align*}
f_{\nu}(z;\lambda_{i,j},\tau_{i,j})
=
\dfrac{1}{F_\nu\left(\tau_{i,j}\Bigl/ \sqrt{1+\lambda_{i,j}^2}\, \right)}\,
f_\nu(z)
F_{\nu+1}
\left(
\sqrt{\nu+1 \over \nu+z^2}\, (\tau_{i,j}+\lambda_{i,j}z)
\right)
\end{align*}
being the Arellano-Valle and Genton's extended skewed Student's $t$-distribution \cite[see][]{Avg:10}.
Now, as $f_{\nu}(\cdot;\lambda_{i,j},\tau_{i,j})$ is a PDF and because \cite[see Proposition 7 of][]{Avg:10}
\begin{align*}
\int_{-\infty}^\infty
z
f_{\nu}(z;\lambda_{i,j},\tau_{i,j})
{\rm d}z
=
{\lambda_{i,j}\over \sqrt{1+\lambda_{i,j}^2}}\,
{\nu\over\nu-1}
\left(1+\dfrac{\tau_{i,j}^2/(1+\lambda_{i,j}^2)}{\nu}\right)
\dfrac{f_\nu\left(\tau_{i,j}\Bigl/ \sqrt{1+\lambda_{i,j}^2}\, \right)}{F_\nu\left(\tau_{i,j}\Bigl/ \sqrt{1+\lambda_{i,j}^2}\, \right)},
\quad \nu>1,
\end{align*}	
$\mu_{H_{i,j}}$ in \eqref{int-uH} can be given as
{\small
	\begin{align*}
	\mu_{H_{i,j}}
	=
	{1\over R_{i,j}}
	\left[
	{\dfrac{\sigma_j}{c_{i,j}} \left(\dfrac{\sigma_j}{\sigma_i}-\rho_{i,j}\right)}\,
	{\nu\over\nu-1}
	\left(
	1+
	\dfrac{1}{\nu c_{i,j}^2}
	\left(\dfrac{\mu_j-\mu_i}{\sigma_i}\right)^2
	\right)
	f_\nu\left(\dfrac{1}{c_{i,j}}\left(\dfrac{\mu_j-\mu_i}{\sigma_i}\right) \right)
	+
	\mu_j 
	F_\nu\left(\dfrac{1}{c_{i,j}}\left(\dfrac{\mu_j-\mu_i}{\sigma_i}\right)\right)
	\right],
	\end{align*}
}\noindent
where
$
c_{i,j}=\sqrt{1-\rho_{i,j}^2+[({\sigma_j}/{\sigma_i})-\rho_{i,j}]^2}.
$

Finally, by applying Theorem \ref{main-theorem} with $\mu_{H_{i,j}}$ as above, we obtain
\begin{eqnarray}\label{GMD-Student-geral}
\resizebox{15.8cm}{!}{$
	\begin{array}{lllll}
	GMD_n
	&=
	\dfrac{2}{\displaystyle\binom{n}{2}} \displaystyle\sum_{1\leqslant i<j\leqslant n}
	\left[
	{\dfrac{\sigma_j}{c_{i,j}} \left(\dfrac{\sigma_j}{\sigma_i}-\rho_{i,j}\right)}\,
	{\nu\over\nu-1}
	\left(
	1+
	\dfrac{1}{\nu c_{i,j}^2}
	\left(\dfrac{\mu_j-\mu_i}{\sigma_i}\right)^2
	\right)
	f_\nu\left(\dfrac{1}{c_{i,j}}\left(\dfrac{\mu_j-\mu_i}{\sigma_i}\right) \right)
	+
	\mu_j 
	F_\nu\left(\dfrac{1}{c_{i,j}}\left(\dfrac{\mu_j-\mu_i}{\sigma_i}\right)\right)
	-
	{\mu_j\over 2}
	\right]
	\\[1,1cm]
	&+
	\dfrac{2}{\displaystyle\binom{n}{2}} \displaystyle\sum_{1\leqslant i<j\leqslant n}
	\left[
	{\dfrac{\sigma_i}{c_{j,i}} \left(\dfrac{\sigma_i}{\sigma_j}-\rho_{j,i}\right)}\,
	{\nu\over\nu-1}
	\left(
	1+
	\dfrac{1}{\nu c_{j,i}^2}
	\left(\dfrac{\mu_i-\mu_j}{\sigma_j}\right)^2
	\right)
	f_\nu\left(\dfrac{1}{c_{j,i}}\left(\dfrac{\mu_i-\mu_j}{\sigma_j}\right) \right)
	+
	\mu_i 
	F_\nu\left(\dfrac{1}{c_{j,i}}\left(\dfrac{\mu_i-\mu_j}{\sigma_j}\right)\right)
	-
	{\mu_i\over 2}
	\right].
	\end{array}
	$}
\end{eqnarray}

Observe that the GMD corresponding to the multivariate Student-$t$ in \eqref{GMD-Student-geral} converges to the GMD of multivariate normal law in \eqref{GMD-Gaussian-geral} as $\nu\to\infty$, which is to be expected.

\begin{proposition}\label{GMD-student-t}
	Let $\boldsymbol{X}\sim EC_n(\boldsymbol{\mu}, \boldsymbol{\Sigma}, g^{(n)})$ be an exchangeable random vector, where $g^{(n)}(x) = (1+x/\nu)^{-(\nu+n)/2}$ and $\nu>1$. 
	Then, the GMD is given by
	\begin{align*}
	GMD_n
	=
	{2\over\sqrt{\pi}}\, \sigma_1 \,
	\dfrac{\sqrt{2\nu}\,\Gamma({\nu+1\over 2})}{(\nu-1)\Gamma({\nu\over 2})} \,
	\dfrac{1}{\displaystyle\binom{n}{2}} 
	\sum_{1\leqslant i<j\leqslant n}\sqrt{1-\rho_{i,j}},
	\end{align*}
	where $\mu_1=\mathbb{E}(X_1)$ and $\sigma_1^2={{\rm Var}(X_1)}$.
\end{proposition}
\begin{proof}
	As $\boldsymbol{X}$	is exchangeable, we have $\mu_i=\mu_j=\mu_1$, $\sigma_i=\sigma_j=\sigma_1$, $\rho_{i,j}=\rho_{j,i}$, $R_{i,j}=R_{j,i}$  and
	$c_{i,j}=c_{j,i}=\sqrt{2(1-\rho_{i,j})}$. Then, the required result follows directly from \eqref{GMD-Student-geral}.
\end{proof}

\paragraph{Acknowledgements}
Roberto Vila  and Helton Saulo gratefully acknowledge financial support from CNPq, CAPES and FAP-DF, Brazil.

\paragraph{Disclosure statement}
There are no conflicts of interest to disclose.

\end{document}